\documentclass{amsart}
\usepackage{graphicx}
\usepackage{todonotes}
\usepackage{rotating}
\usepackage{mathtools}
\usepackage[utf8]{inputenc}
\usepackage{ mathrsfs }
\usepackage{caption}
\usepackage{amssymb}
\usepackage{amsthm,amsmath}
\usepackage{todonotes}
\usepackage{tikz}
\usepackage{enumitem}
\usetikzlibrary{decorations.pathreplacing,calligraphy}
\usepackage{tikz}
\usepackage{pgfplots}
\pgfplotsset{compat=1.18}
\usepgfplotslibrary{groupplots}
\newtheorem{theorem}{Theorem}
\newtheorem{proposition}{Proposition}
\newtheorem{defn}{Definition}

\newtheorem{corollary}{Corollary}
\newtheorem{lemma}{Lemma}

\author{Anton Gorodetski}

\address{Department of Mathematics, University of California, Irvine CA 92697, USA}

\email{asgor@uci.edu}

 \author{Alexandro Luna}

 \address{Department of Mathematics, The Ohio State University, Columbus OH 43210, USA}

 \email{luna.165@osu.edu}

\author{William Wood}

\address{Department of Mathematics, University of California, Riverside CA 92507, USA}

\email{wwood@ucr.edu}

\title{On sums of ``small'' Cantor sets}
\begin{document}

\begin{abstract} \noindent We provide several results showing that sums of ``small'' Cantor sets can turn into an interval or a finite union of intervals. Specifically, we show that 

1) There exists a Cantor set $C\subset [0,1]$ of zero logarithmic capacity such that $C+C=[0,2]$. 

2) For any $n\ge 2$ there exists a Cantor set $C\subset [0,1]$ of zero logarithmic capacity such that a Minkowski sum of $n$ copies of $C$ is an interval $[0,n]$, while a sum of any smaller number of copies of $C$ is still a Cantor set. 

3) For any $n\ge 2$ there exists a collection of Cantor sets $C_1, C_2, \ldots, C_n\subset [0,1]$ of zero lower box-counting dimension 
such that a Minkowski sum of any $n-1$ sets from this collection is a Cantor set of zero lower box-counting dimension, while the sum of all $n$ sets is  an interval $[0,n]$.

4) There is an uncountable family of Cantor sets of zero lower box-counting dimension such that the sum of any two different sets from this family is a finite union of intervals. 

5) For each of the properties above, there are examples of Cantor sets with the same property having zero Hausdorff $h$-measure for an arbitrary given gauge function $h$. 
\end{abstract}

\maketitle

\section{Introduction}
 Questions on topological structure, measure, and fractal dimensions of sums of Cantor sets appear in many areas of mathematics -- dynamical systems \cite{Gor, N, N79, PT}, Diophantine approximations \cite{CF, H, Ma, MM, Mor, Mor18, MorICM}, harmonic analysis \cite{BKMP, BM}, spectral theory \cite{DG18, DGS, EL}, etc. In many cases the results are presented in the following form: if Cantor sets are ``large'' in some sense (in terms of thickness, measure, some kind of fractal dimension, etc.), then their Minkowski sum, sometimes under genericity assumptions,  is ``large'' in some other way (e.g. has maximal possible fractal dimension as in \cite{HS}, or has positive Lebesgue measure as in \cite{DGS, PSol, SSol, Sol97, Sol23}, or contains an interval as in \cite{N, N79}). For example, Steinhaus' Theorem claims that the sum of two Cantor sets of positive Lebesgue measure must contain an interval \cite{St}. One of the most famous results of this sort is the celebrated Newhouse's Gap Lemma \cite{N, N79}, which claims that if the product of thicknesses of two Cantor sets is greater than one, then its sum must be a finite union of intervals.\footnote{The classical Newhouse's Gap Lemma  deals with intersections of Cantor sets; we formulate it in terms of sumsets to put it into the context of this paper.} An example of two Cantor sets of zero thickness whose sumset is an interval is constructed in \cite{HO}. Extension of the Newhouse's Gap Lemma to sums of more than two Cantor sets was obtained by Astels \cite[Theorem 2.4]{A} (see also \cite{A2, A3, HO}).

Numerous results were obtained on sums of dynamically defined Cantor sets. J.\,Palis conjectured \cite {P} that generically the sum of two dynamically defined Cantor sets either contain an interval, or has zero Lebesgue measure. The conjecture was proved by Moreira and Yoccoz  \cite{MorY}. They showed that generically, if the sum of Hausdorff dimensions of two Cantor sets is larger than one, then their sum contains an interval. This result has numerous consequences. An analogous conjecture for affine Cantor sets is still open (but see \cite{T, T19} for closely related statements).

Dynamically defined Cantor sets, including affine Cantor sets, while must have zero Lebesgue measure, are ``not too small''; in particular, they must have positive Hausdorff dimension. At the same time ``small'' Cantor sets, for example, Cantor sets of zero Hausdorff dimension, naturally appear in some models in mathematical physics and spectral theory, e.g. see \cite{DFG},  \cite{DFL}, \cite{EFGL}, or \cite{FGH}. Therefore, it  makes perfect sense to ask an opposite question: if two Cantor sets are ``small'' in some sense, does it imply that their sum is also ``small'', possibly in a different sense? 
 One example of such result is the following (e.g. see \cite[Chapter 4]{PT}): if sum of box counting dimensions of two Cantor sets is smaller than one, then their sum must be a Cantor set of zero Lebesgue measure.  While certainly useful  (e.g. see \cite{DFG}), this result is formulated in terms of  box counting dimension, and in many natural cases ``small''\footnote{ Sometimes the term ``thin'' is used in relation to spectra of some Schr\"odinger operators, e.g. see \cite{DF, DFW, FGH}, but this clashes with the notion of a thin set in potential theory, so we will not do that.} Cantor sets do not have a well defined box counting dimension.  So, it is natural to ask this question in terms of small (or zero) Hausdorff dimension, or small (or zero) logarithmic capacity, or small (or zero) lower box counting dimension. 

 Moreover, the questions about sums of ``small'' Cantor sets appeared in a natural way in several models in spectral theory of ergodic Schr\"odiger operators. Indeed, it was shown in \cite{DFW} that for a dense set of limit periodic potentials the spectrum of the corresponding Schr\"odinger operator is a Cantor set of zero lower box counting dimension. Analogous statement for limit periodic Schr\"odinger operators on the real line was obtained in \cite{DFL}. Using these results, one can construct almost periodic potentials on the plane (or higher dimensional lattice) that have zero lower box counting dimension, e.g. see \cite{DFG}. The construction uses the fact that for separable potentials the spectrum is given by the Minkowski sum of the spectra of operators on the line (or one dimensional lattice, in the discrete case).  As another example, one can consider the spectrum of the Almost Mathieu operator at critical coupling. It is known that for a generic frequency its spectrum has zero Hausdorff dimension \cite{LS}. At the same time, it is not known whether the spectrum of the corresponding operator on the lattice with separable potential (or, equivalently, the sum of the spectrum of critical Almost Mathieu operator  with itself) is a Cantor set or contains an interval; see \cite{B, Ta} for related questions in the case of non-critical coupling.


In this paper we show that no degree of ``smallness'' of Cantor sets in terms of Hausdorff dimension, logarithmic capacity, or any type of generalized Hausdorff measure can guarantee that their sum is ``small'' in any way. Let us now formally state our results. 

\subsection{Sums of sets of  zero logarithmic capacity}

Let $C \subset \mathbb{R}$ be a compact set.  Denote by $\mathcal{M}(C)$ the set of
Borel probability measures supported on $C$.  For $\mu \in \mathcal{M}(C)$,
define the \emph{logarithmic energy}
\[
  I(\mu) \;=\; \iint \log \frac{1}{|z - w|}\, d\mu(z)\, d\mu(w)
  \;\in\; (-\infty,\, +\infty].
\]
 
\begin{defn}[Logarithmic capacity]\label{def:log_cap}
The {logarithmic capacity} of $C$ is
\[
  \mathrm{cap}(C) \;=\; e^{-V(C)},
\]
where
\[
  V(C) \;=\; \inf_{\mu \in \mathcal{M}(C)} I(\mu).
\]
If $V(C) = +\infty$, one sets 
$\mathrm{cap}(C) = 0$.
\end{defn}

This is not the only way to define the logarithmic capacity of a set; other equivalent definitions involve  transfinite diameter of a set or Green's function, e.g. see \cite{R} for details. 
 
A compact set with
$\mathrm{cap}(C) = 0$ is called a \emph{polar set}; polar sets form a collection of 
negligible sets in  logarithmic potential theory.

\begin{theorem}\label{t.first}
There exists a Cantor set $C\subseteq [0,1]$ of zero logarithmic capacity such that $C+C=[0,2]$.
\end{theorem}

Moreover, one can improve this statement in the following way:

\begin{theorem}\label{t.2}
For any $n\ge 2$ there exists a Cantor set $C\subseteq [0,1]$ of zero logarithmic capacity such that
$$
 \underbrace{C+C+\cdots +C}_{n\text{ times}}\equiv\{x_1+x_2+\ldots+x_n\ |\ x_i\in C, \ i=1,2, \ldots, n\}=[0,n],
$$
while for any $m<n$ the sum $$\underbrace{C+C+\cdots +C}_{m\text{ times}}$$ is a Cantor set. 
\end{theorem}

    A Cantor set  of zero logarithmic capacity must have zero Hausdorff dimension.  Notice that in the case when logarithmic capacity in Theorem \ref{t.2} is replaced  by Hausdorff dimension, the similar questions were studied by Schmeling and Shmerkin   \cite{SS}:

\begin{theorem}[\cite{SS}]
For any nondecreasing sequence $\{\alpha_k\}_{k\in \mathbb{N}}$, with $\alpha_k$ taking values in $[0, 1]$, there
exists a Cantor set $C$ such that 
$$
\text{\rm dim}_H \underbrace{C+C+\cdots +C}_{k\text{ times}}=\alpha_k
$$
for each $k\in \mathbb{N}$. Moreover, if $\alpha_k=1$ for some $k$, one
can also require that
the corresponding sumset contains an interval.
\end{theorem}
In fact, \cite{SS} contains even stronger result that allows to prescribe not only Hausdorff dimension of the sumsets, but also lower and upper box counting dimensions, modulo some explicit quantitative restrictions.  

It is interesting to notice that spectra of ergodic Schr\"odinger operators cannot have zero logarithmic capacity \cite{S}, but can have zero lower box counting dimension \cite{DF, DFG}. This motivates the next series of results:

\subsection{Sums of sets of zero lower box counting dimension}

Notice that one cannot just replace a set of zero Hausdorff dimension in Theorem \ref{t.first} and Theorem \ref{t.2} by a set of zero lower box counting dimension. Indeed, the Minkowski sum of any number of copies of a Cantor set of zero lower box counting dimension must be a Cantor set of zero lower box counting dimension. Nevertheless, the following statement holds:

\begin{theorem}\label{t: multi-sum lower box}
    Let $n\geq 2$. 
    There exist $n$-many Cantor sets $C_1,\dots, C_{n}$ such that 
\begin{equation}
   \underline{\dim}_B\left(C_{i_1}+\cdots +C_{i_k}\right)=0
\end{equation}
for all $k<n$ and $(i_1,\dots, i_k)\in \{1,\dots, n\}^k$, 
and yet 
\begin{equation}
    C_1+\cdots+ C_{n}=[0,n].
\end{equation}
\end{theorem}

Also, it is not hard to present a finite collection of Cantor sets of zero lower box counting dimension such that the sum of any two of them is a finite union of intervals. With a bit of effort, one can construct an infinite collection of Cantor sets with the same property. What is more surprising is that it is possible to find an {\it uncountable} collection of such Cantor sets:

\begin{theorem}\label{t: uncountable lower box}
    There is an uncountable collection of Cantor sets $\mathcal C$ of $[0,1]$ such that the following hold:
\begin{itemize}
    \item $\underline{\dim}_B(C)=0$ for all $C\in \mathcal C$
    \item $C+C'$ is a union of finitely many  intervals
   for any $C, C'\in \mathcal C$ with $C\neq C'$.
\end{itemize}
\end{theorem}

\subsection{Sums of sets of zero generalized Hausdorff measure}\label{ss.Haus}

We can also show that in some sense no degree of ``smallness'' in terms of generalized Hausdorff measure of Cantor sets can guarantee ``smallness'' of their sum. Namely, let $h:[0, +\infty)\to [0, +\infty)$ be any continuous monotonically increasing function with $h(0)=0$. We will call such a function ``gauge function'', following the terminology from \cite{M}; notice that the terms ``measuring function'' and ``dimension function'' have been used in the literature as well, e.g. see \cite{AG} or \cite{GMS}. For a given gauge function $h$, a Cantor set $C\subseteq [0,1]$, and $\delta>0$ one sets
\[\mu_\delta^{h}(C)=\inf\bigg\{\sum\limits_i h(|U_i|):\bigcup\limits_i U_i\supseteq C,\;U_i\text{ is an interval},|U_i|<\delta\bigg\},\]
and defines the Hausdorff $h$-measure of $C$ as 
\[\ \mu^{h}(C)=\lim\limits_{\delta\rightarrow0}\mu_\delta^h(C). \]
For example, if $h(x)=x^d$, then $\mu^{h}(C)$ is the $d$-dimensional Hausdorff measure of $C$. 

There is an interesting interplay between the logarithmic capacity of a set $C$ and its Hausdorff $h$-measures for different classes of the gauge functions. For example, if $\mathrm{cap}(C) = 0$ and $\int_0^1\frac{h(t)}{t}dt<\infty$, then $\mu^h(C)=0$, e.g. see \cite{Fr}. In the opposite direction, Ursell proved \cite{U} that if $\liminf_{t\to 0}h(t)|\log t|=0$, then there exists a set $C=C_h$ such that $\mu^h(C)<\infty$ and $\mathrm{cap}(C) >0$. See \cite[Chapter 3]{Q2} for a detailed history of the question, and \cite{KQ, Q1} for some related results for non-compact sets.

It turns out that in all of our results stated above one can require the Cantor sets to have zero Hausdorff $h$-measure with respect to an arbitrary given in advance gauge function $h$:
\begin{theorem}\label{t.all}
For any given  gauge function $h:[0, +\infty)\to [0, +\infty)$ the following statements hold:
\begin{enumerate}[label=(\Roman*)]
\item For any $n\ge 2$, there exists a Cantor set $C\subseteq [0,1]$ such that $\mu^h(C)=0$, and 
$$
 \underbrace{C+C+\cdots +C}_{n\text{ times}}\text{ is an interval.}
$$
while for any $m<n$ the sum $$\underbrace{C+C+\cdots +C}_{m\text{ times}}$$ is a Cantor set. 
\item For any $n\geq 2$, there exist $n$-many Cantor sets $C_1,\dots, C_{n}$ such that $\mu^h(C_i)=0$ for each $i=1, \ldots, n,$ and
\begin{equation}      \underline{\dim}_B\left(C_{i_1}+\cdots +C_{i_k}\right)=0 
\end{equation}
for all $k<n$ and $(i_1,\dots, i_k)\in \{1,\dots, n\}^k$, 
and yet 
\begin{equation}
    C_1+\cdots+ C_{n}=[0,n].
\end{equation}
\item There is an uncountable collection of Cantor sets $\mathcal C$ of $[0,1]$ such that the following hold:
\begin{itemize}
    \item $\underline{\dim}_B(C)=0$ and $\mu^h(C)=0$ for all $C\in \mathcal C$;
    \item $C+C'$ is a union of finitely many  intervals
   for any $C, C'\in \mathcal C$ with $C\neq C'$.
\end{itemize}
\end{enumerate}
\end{theorem}

\section*{Acknowledgements}
We would like to thank David Damanik for bringing our attention to the question. The project was partly supported by NSF grant DMS-2247966.
 
\section{Sums of Cantor Sets with Zero Lower Box-counting Dimension}

We begin by proving parts (II) and (III) of Theorem \ref{t.all}. The rest of the results will follow from these parts of the theorem, and we conclude their proofs in Section \ref{Section: Concluding the Proofs}. For the rest of this section, we will assume that $n_0\geq 2$ and that $h:[0,\infty)\rightarrow[0,\infty)$ is an arbitrary gauge function. 

This section is outlined as follows: Subsections \ref{ss: Linear Expanding Maps}, \ref{ss: A Preliminary Sequence}, and \ref{ss: Construction of the Cantor Sets} are used to develop the Cantor sets of interest. Subsections \ref{ss: Linear Expanding Maps} and \ref{ss: A Preliminary Sequence} are dedicated to technical conditions and definitions, and the main family of Cantor sets is defined in Subsection \ref{ss: Construction of the Cantor Sets}. Within, Subsection \ref{sss: Main Properties of the Cantor Sets} is used to prove technical properties of the Cantor sets, which will ultimately be used in the proofs of the main theorems. Then, part (II) of Theorem \ref{t.all} is proved in Subsection \ref{ss: Multi-sums of Cantor Sets} and part (III) is proved in Subsection \ref{ss: An Uncountable Family of Zero Dimensional Cantor Sets}. 

\subsection{Linear Expanding Maps}\label{ss: Linear Expanding Maps}
We work with Cantor sets given as attractors of non-stationary systems of linear expanding maps. 

For a fixed parameter $d > 2$, we define the map $E_d : [0, \frac{1}{d}] \cup [1 - \frac{1}{d}, 1] \to [0,1]$ by
\begin{equation}\label{e.Ed}
E_d(x) := \begin{cases}
dx, & x \in [0, \frac{1}{d}], \\
d(x - 1) + 1, & x \in [1 - \frac{1}{d}, 1].
\end{cases}
\end{equation}

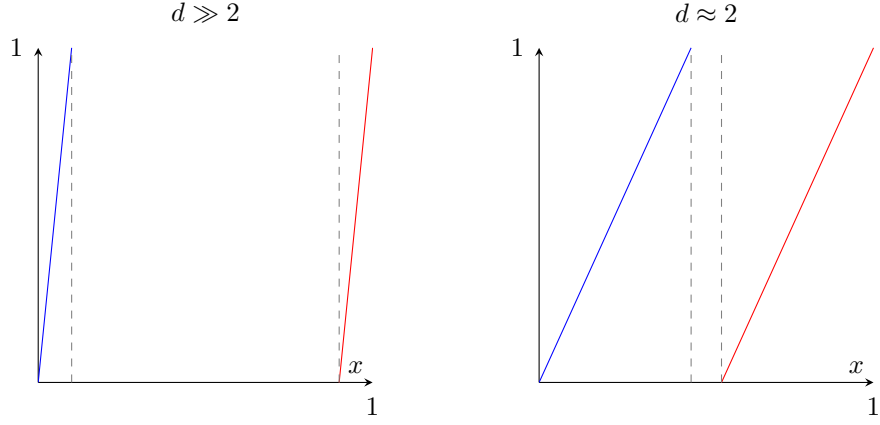
\begin{figure}[htbp]
\centering
\begin{tikzpicture}
\begin{groupplot}[
    group style={
        group size=2 by 1,
        horizontal sep=2.2cm,
    },
    width=6cm,
    height=6cm,
    xmin=0, xmax=1,
    ymin=0, ymax=1,
    axis lines=middle,
    xlabel={$x$},
    xtick={0,1},
    ytick={0,1},
    tick style={draw=none},
    samples=8,
    domain=0:1]

\nextgroupplot[title={$d \gg 2$}]
\addplot[blue, domain=0:0.1] {10*x};
\addplot[red, domain=0.9:1] {10*(x-1)+1};
\draw[dashed, gray] (axis cs:0.1,0) -- (axis cs:0.1,1);
\draw[dashed, gray] (axis cs:0.9,0) -- (axis cs:0.9,1);

\nextgroupplot[title={$d \approx 2$}]
\addplot[blue, domain=0:{1/2.2}] {2.2*x};
\addplot[red, domain={1-1/2.2}:1] {2.2*(x-1)+1};
\draw[dashed, gray] (axis cs:{1/2.2},0) -- (axis cs:{1/2.2},1);
\draw[dashed, gray] (axis cs:{1-1/2.2},0) -- (axis cs:{1-1/2.2},1);

\end{groupplot}
\end{tikzpicture}
\caption{The map $E_d$ for $d=10$ (left) and $d=2.2$ (right).}
\label{fig:Ed-graphs}
\end{figure}

Given a block (or word) $D=d_1\dots d_l\in (2,\infty)^l$, we define 
\begin{equation}
    \mathrm{prod}(D):=\prod_{i=1}^\ell d_{i},
\end{equation}
\begin{equation}
    \mathrm{length}(D):=l,
\end{equation}
and
\begin{equation}
    E_D^{-1}:=E_{d_1}^{-1}\circ\cdots \circ E_{d_l}^{-1}.
\end{equation}
Notice that for any collection of words $D_1,\dots, D_k$, we have that 
\begin{equation}
    E_{D_1}^{-1}\circ\cdots\circ E_{D_k}^{-1}([0,1])
\end{equation}
is a union of $2^{\mathrm{length}(D_1)+\cdots+\mathrm{length}(D_k)}$ many disjoint intervals in $[0,1]$, each of length  $\prod_{i=1}^k\mathrm{prod}\left(D_i\right)^{-1}$, and this set contains $0$ and $1$. 
Also, for any sequence of words $D_\infty:=(D_1,D_2,\cdots)$, we set 
\begin{equation}
    C_k\left(D_\infty\right):= E^{-1}_{D_1}\circ\cdots \circ E^{-1}_{D_k}([0,1]),
\end{equation}
and 
\begin{equation}
C\left(D_\infty\right):=\bigcap_{k=1}^\infty  C_k\left(D_\infty\right).
\end{equation}

\subsection{A Preliminary Sequence.}\label{ss: A Preliminary Sequence}
We now set up exactly which blocks to consider, when iterating our expanding maps. 

Let us start with an informal description of the construction of the blocks. We will choose two sequences $\overline a, \overline b\in (2,\infty)^{\mathbb N}$ so that

\begin{itemize}
    \item $a_k\rightarrow 2$ and
    \item $b_k\rightarrow \infty$ as $k\to \infty.$ 
\end{itemize}


Then, we will consider Cantor sets $C\left(D_{\infty}\right)$ where $D_\infty=(D_1,D_2,\dots)$ is a sequence of blocks so that
\begin{equation}
    D_k=a_k\dots a_k \ \text{or} \ D_k=b_k\dots b_k,
\end{equation}
and blocks of each type occur infinitely many times. The lengths of the blocks will be given explicitly, but grow in $k$. Roughly speaking, when $D_k=a_k\dots a_k$, the set $C_{k}(D_\infty)$
is ``thick", and when $D_k=b_k\dots b_k$, the set $C_{k}(D_\infty)$ is ``thin." Switching between the two selections allows us to construct a Cantor set that is the intersection of a sequence of nested sets, where this sequence of sets has infinitely many thick elements and also infinitely many thin elements. Under suitable conditions, the presence of the thick sets will give us our desired sum-set as an interval,
and the presence of the thin sets will guarantee that $C(D_\infty)$ is zero-dimensional (and has zero $h$-measure). 

We now give the precise conditions on the sequences $\overline a$ and $\overline b$. Let us recall that a gauge function $h$ is chosen and fixed now, and so is $n_0\in\mathbb N$. 
\begin{lemma}\label{lemma: sequence conditions}
 
    There exist sequences $\overline a, \overline b, \overline c\in (2,\infty)^{\mathbb N}$ and $(\ell_k)\in \mathbb N^{\mathbb N}$, so that the following hold:
\begin{itemize}
\item $\overline b$ is strictly increasing and $b_n\rightarrow \infty$
\item For every $k\geq 1$, 
\begin{equation}\label{e: thickness condition for sequences}
\frac{1}{a_k-2}\cdot \frac{1}{a_k-2}>1, \   \frac{1}{a_k-2}\cdot \frac{1}{b_k-2}>1, \ \text{and}\  \ \frac{1}{a_k-2}\cdot \frac{1}{c_k-2}>1
\end{equation}
\item For any $k\geq 1$, the following equality holds
\begin{equation}\label{e: products equal}
a_k^{\ell_{2k-1}+\ell_{2k}}=b_k^{\ell_{2k-1}}c_k^{\ell_{2k}}
\end{equation}
\item We have 
\begin{equation}\label{e: conditions on lengths of blocks}
\lim\limits_{k\rightarrow \infty}  \frac{\ell_{2k-1}}{\ell_1+\cdots+\ell_{2k-1}} =1 
\end{equation}
\item We have 
\begin{equation}\label{e: gauge condition in sequqnce construction}
\lim\limits_{k\rightarrow\infty}2^{(n_0-1)\left(\ell_1+\cdots+\ell_{2k-1}\right)}h\left(\frac{n_0-1}{b_k}\right)=0.
\end{equation}
\end{itemize}
\end{lemma}

\begin{proof}
 The sequences can be constructed  by induction in $k$, by choosing each parameter in the sequence in a particular order. Namely, suppose that we had defined up to terms $a_{k-1}$, $b_{k-1}$, $c_{k-1}$,  $\ell_{2(k-1)-1}$, and $\ell_{2(k-1)}$. We can choose $\ell_{2k-1}$ so large that 
 \begin{equation}
   \frac{\ell_{2k-1}}{\ell_1+\cdots+\ell_{2k-1}}  >1-\frac{1}{k}.
 \end{equation}
 Then, choose $b_k>b_{k-1}$ so large that 
 \begin{equation}
   2^{(n_0-1)\left(\ell_1+\cdots+\ell_{2k-1}\right)}h\left(\frac{n_0-1}{b_k}\right)  <\frac{1}{k},
 \end{equation}
and then choose $a_k>2$ close enough to $2$ so that $ \frac{1}{a_k-2}\frac{1}{b_k-2}>1.$
Let us choose $c_k'\in (2,a_k)$ close enough to $a_k$ so that 
$$\left(\frac{c_k'}{a_k}\right)^{\ell_{2k-1}}>\left(\frac{a_k}{b_k}\right)^{\ell_{2k-1}}.$$
Then, since $\left(\frac{c_k'}{a_k}\right)^t\rightarrow 0$ as $t\rightarrow \infty$, there is some $t^*>\ell_{2k-1}$ such that 
\begin{equation}
    \left(\frac{c_k'}{a_k}\right)^{t^*}=\left(\frac{a_k}{b_k}\right)^{\ell_{2k-1}}.
\end{equation}
Set $\ell_{2k}:=\left \lceil t^* \right \rceil$ and set $c_k:=a_k\left( \frac{c'_k}{a_k}\right)^{\frac{t^*}{\ell_{2k}}}.$ Then, 

$$\left( \frac{c_k}{a_k}\right)^{\ell_{2k}}=\left(\frac{c_k'}{a_k}\right)^{t^*}=\left(\frac{a_k}{b_k}\right)^{\ell_{2k-1}},$$
so that (\ref{e: products equal}) holds. Moreover, since $\frac{t^*}{\ell_{2k}}\leq 1$ and $\frac{c'_k}{a_k}<1$, we have 
$$c_k=a_k\left( \frac{c'_k}{a_k}\right)^{\frac{t^*}{\ell_{2k}}}\geq a_k\cdot \frac{c'_k}{a_k}>2.$$
\end{proof}

The reason for incorporating the sequence $\overline c$ will be apparent later in the technical proofs, but roughly speaking, it is used as a renormalization tool. Namely, for any two sequences of blocks $D_\infty=\left(D_1,D_2,\dots\right)$ and $D_\infty'=\left(D_1', D_2',\dots\right)$, with 
\begin{equation}
    D_k= \underbrace{a_k\dots a_k}_{\ell_{2k-1}+\ell_{2k}} \ \text{or} \ \underbrace{b_k\dots b_k}_{\ell_{2k-1}}\underbrace{c_k\dots c_k}_{\ell_{2k}}, \ \text{and} \ D_k'= \underbrace{a_k\dots a_k}_{\ell_{2k-1}+\ell_{2k}} \ \text{or} \ \underbrace{b_k\dots b_k}_{\ell_{2k-1}}\underbrace{c_k\dots c_k}_{\ell_{2k}},
\end{equation}
for any $k$, the components of the sets $C_k\left(D_\infty\right)$ and $C_k\left(D_\infty '\right)$ are all closed intervals of equal length. This will be particularly helpful when trying to show that the sum-set $C\left(D_\infty\right)+C\left(D_\infty '\right)$ contains an interval, under suitable conditions on $D_\infty$ and $D_\infty'$. However, the presence of longs blocks of $b_k$'s still guarantees the desired zero-dimensional (and zero $h$-measure) outcomes. 
 
\subsection{Construction of the Cantor Sets}\label{ss: Construction of the Cantor Sets}
We now construct the Cantor sets of interest using the sequences provided by Lemma \ref{lemma: sequence conditions}. 

Given $\omega\in \{0,1\}^{\mathbb N}$ and $k\in\mathbb N$, we consider the block
\begin{equation}
    B_k(\omega):= \begin{cases} 
     \underbrace{a_k\dots a_k}_{\ell_{2k-1}+\ell_{2k}} & \text{if} \ \omega_k=1 \\
   \underbrace{b_k\dots b_k}_{\ell_{2k-1}}\underbrace{c_k\dots c_k}_{\ell_{2k}} & \text{if} \ \omega_k=0
   \end{cases},
\end{equation}
and define the Cantor set 
\begin{equation}
    C(\omega):=\bigcap_{k=1}^{\infty} C_k(\omega),
\end{equation}
where 
\begin{equation}
  C_k(\omega):=  E^{-1}_{B_1(\omega)}\cdots E^{-1}_{B_k(\omega)}([0,1]).
\end{equation}
It should be noted that $C_{k+1}(\omega)\subset C_{k}(\omega)$. For notational convenience, we also set 
\begin{equation}
    C_{[k:k']}(\omega):= E_{B_{k'+1}(\omega)}^{-1}\cdots E_{B_{k}(\omega)}^{-1}([0,1]),
\end{equation}
for $0\leq k'< k$, and 
\begin{equation}
    C_{[\infty: k]}:=\bigcap_{q=k+1}^{\infty} E^{-1}_{B_{k+1}(\omega)}\cdots E^{-1}_{B_q(\omega)}([0,1]).
\end{equation}
We now investigate the properties of these Cantor sets. 

\subsubsection{Main Properties of the Cantor Sets}\label{sss: Main Properties of the Cantor Sets}

Here, we develop two technical propositions that will be used in the proofs of (II) and (III) of Theorem \ref{t.all}. The proofs of these parts of Theorem \ref{t.all} will be concluded in  Subsections \ref{ss: Multi-sums of Cantor Sets} and \ref{ss: An Uncountable Family of Zero Dimensional Cantor Sets}, respectively. We now state and prove the propositions.

First, we investigate conditions  on our underlying sequences sufficient to ensure that an $n$-sum of such Cantor sets contains an interval.

\begin{proposition}\label{prop: n sum is a collction of intervals}
Suppose that $n\in\mathbb N$ and $\omega^{(i)}$, $i=1,\dots, n$, are sequences such that for some $k_0\geq 0$, we have 
\begin{equation}
    \omega^{(1)}_k=\cdots=\omega^{(n)}_k
\end{equation}
for all $k\leq k_0$ and 
\begin{equation}
    \omega^{(1)}_k+\cdots+ \omega^{(n)}_k\geq 1,
\end{equation}
for all $k>k_0$. Then, $C\left(\omega^{(1)}\right)+\cdots+ C\left(\omega^{(n)}\right)$ is a union of finitely many disjoint intervals. Moreover, if $k_0=0$, then
\begin{equation}\label{e: sum of n sets is interval - LB - Proposition}
C\left(\omega^{(1)}\right)+\cdots+ C\left(\omega^{(n)}\right)=[0,n].
\end{equation}
\end{proposition}

Before, presenting the proof, we recall the notion of \textit{thickness} for compact sets of $\mathbb R$, which will be a tool used in the proof. 

Given a compact set $C\subset \mathbb R$ with convex hull $I$, a \textit{gap} of $C$ is a component of $I\setminus C$. Any enumeration $\mathcal U=\{U_n\}$ of the gaps of $C$ is called a \textit{presentation} of $C$. Given a boundary point $u\in C$ of some gap $U_k$, we call the connected component $B$ contained in $I\setminus\left(U_1\cup\dots\cup U_k\right) $ that contains $u$ a \textit{bridge} of $U_k$. We define 
\begin{equation}
    \mathcal T(C, \mathcal U, u):=\frac{|B|}{|U_k|},
\end{equation}
and set the \textit{thickness} of $C$ to be 
\begin{equation}
    \mathcal T(C):=\sup\limits_{\mathcal U}\inf\limits_{u}\mathcal T(\mathcal C, \mathcal U, u).
\end{equation}
For example, if $d>2$ and $E_d$ is defined by (\ref{e.Ed}), then it is easy to see that for the Cantor set $C=\bigcap_{n=1}^\infty E_d^{-n}([0,1])$ one has $\mathcal T(C)=\frac{1}{d-2}$.

The Newhouse Gap Lemma (initially obtained in \cite{N, N79} to establish the phenomenon of persistent homoclinic tangencies)  states 
\begin{theorem}\label{appendix t: Newhouse Gap Lemma}
    Suppose that $C$ and $C'$ are two compact sets so that 
    \begin{equation}
        \mathcal T(C)\mathcal T\left(C'\right)>1.
    \end{equation}
    Then, one of the following holds:
    \begin{itemize}
        \item $C$ is contained in a gap of $C'$, or vice versa
        \item $C\cap C'\neq \emptyset$
    \end{itemize}
\end{theorem}
As a corollary, we obtain (e.g. see \cite[Lemma 6.2]{DG11} for details)
\begin{corollary}\label{appendix cor: corollary to Gap Lemma}
    Under the assumptions of Theorem \ref{appendix t: Newhouse Gap Lemma}, if the minimum and  maximum values of $C$ are $0$ and $1$, respectively, and the same holds for $C'$, then 
    \begin{equation}
        C+C'=[0,2].
    \end{equation}
\end{corollary}

We now prove Proposition \ref{prop: n sum is a collction of intervals}.
\begin{proof}[Proof of Proposition \ref{prop: n sum is a collction of intervals}.]
First, suppose $k_0=0$. We will show that 
\begin{equation}\label{e: sum of n many level sets is [0,n]}
C_k\left(\omega^{(1)}\right)+\cdots+ C_k\left(\omega^{(n)}\right)   =[0,n],
\end{equation}
for all $k\in\mathbb N$. 

First, for any sequence $\omega$ of $0$'s and $1$'s, we have 
\begin{equation}\label{e: thickness calculation in the proof}
    \mathcal T\left(C_{[k:k-1]}(\omega)\right):=\begin{cases} 
     \frac{1}{a_k-2} & \text{if} \ \omega_k=1 \\
   \min\{\frac{1}{b_k-2}, \frac{1}{c_k-2}  \} & \text{if} \ \omega_k=0
   \end{cases}.
\end{equation}
So, if $\omega'$ is another sequence such that $\omega_k$ and $\omega_k'$ are not both $0$, then 
\begin{equation}\label{e: thickness implication}
    C_{[k:k-1]}(\omega)+C_{[k:k-1]}(\omega')=[0,2],
\end{equation}
by the Newhouse Gap Lemma (specifically, Corollary \ref{appendix cor: corollary to Gap Lemma}), due to assumption (\ref{e: thickness condition for sequences}).

Let $k\in\mathbb N$. Since there is some $i\in\{1,\dots, n\}$ such that $\omega_k^{(i)}=1$,
we get that 
\begin{equation}
    C_{[k:k-1]}\left(\omega^{(i)}\right)+C_{[k:k-1]}\left(\omega^{(j)}\right)=[0,2],
\end{equation}
for any $j\neq i$. Moreover, as $0$ and $1$ are contained in $C_{[k:k-1]}\left(\omega^{(j)}\right)$ for all $j$, we obtain 
\begin{equation}\label{e: thickness implication for multi-sum}
    C_{[k:k-1]}\left(\omega^{(1)}\right)+\cdots+ C_{[k:k-1]}\left(\omega^{(n)}\right)   =[0,n].
\end{equation}
We now proceed to show (\ref{e: sum of n many level sets is [0,n]}) by induction on $k$. First, setting $k=1$ in the previous equation shows that the base of induction holds. 

Now, let $k\ge 2$ be arbitrary, and suppose that $\sum_{i=1}^{n}C_{k-1}\left(\omega^{(i)}\right)=[0,n]$. For each $i=1,\dots, n$, the set $C_{k-1}\left(\omega^{(i)}\right)$ has precisely $2^{\ell_1+\cdots+\ell_{2(k-1)}}$ many disjoint components. These components are all closed interval of equal length and by (\ref{e: products equal}), the lengths are independent of $i$. Denote these common lengths by $a$. Since by definition the set $C_{k-1}\left(\omega^{(i)}\right)$  is the pre-image of $[0,1]$ under the map $E_{B_1\left(\omega^{(i)}\right)}\cdots E_{B_{k-1}\left(\omega^{(i)}\right)}$, each component of $C_{k-1}\left(\omega^{(i)}\right)$ is an image of $[0,1]$ under  a linear contraction map with contraction rate $a$. 

Precisely speaking, for each $i=1,\dots, n$, there are constants $\alpha_l\left(\omega^{(i)}\right)$, $l=1,\dots, 2^{\ell_1+\cdots+\ell_{2(k-1)}},$ such that 
\begin{equation}\label{e: disjoint union of inverse branch images}
    C_{k-1}\left(\omega^{(i)}\right)=\bigsqcup_{l}\left(a[0,1]+\alpha_l\left(\omega^{(i)}\right)\right).
\end{equation}

Let $t\in [0,n]$ be arbitrary. We will show that $t$ is contained in $\sum_{i=1}^{n}C_{k}\left(\omega^{(i)}\right)=[0,n]$. By the induction hypothesis, $t$ is contained in $\sum_{i=1}^{n}C_{k-1}\left(\omega^{(i)}\right)=[0,n]$, so by (\ref{e: disjoint union of inverse branch images}), we have
\begin{equation}
    t\in \sum_{i=1}^{n}\left(a[0,1]+\alpha_{l_i}\left(\omega^{(i)}\right)\right),
\end{equation}
for some indices $l_1,\dots, l_{n}\in\left\{1,\dots, 2^{\ell_1+\cdots+\ell_{2(k-1)}}\right\}$. 

Now, from (\ref{e: thickness implication for multi-sum}) 
we have 
\begin{align*}
\sum_{i=1}^{n}\left(a[0,1]+\alpha_{l_i}\left(\omega^{(i)}\right)\right)&=a[0,n]+\sum_{i=1}^{n}\alpha_{l_i}\left(\omega^{(i)}\right)\\
&=a\sum_{i=1}^{n}C_{[k:k-1]}\left(\omega^{(i)}\right)+\sum_{i=1}^{n}\alpha_{l_i}\left(\omega^{(i)}\right)\\
&=\sum_{i=1}^{n}\left(aC_{[k:k-1]}\left(\omega^{(i)}\right)+\alpha_{l_i}\left(\omega^{(i)}\right)\right).
\end{align*}
However, by definition, the set $C_{k}\left(\omega^{(i)}\right)$ is the pre-image of $C_{[k: k-1]}\left(\omega^{(i)}\right)$ under $E_{B_1\left(\omega^{(i)}\right)}\cdots E_{B_{k-1}\left(\omega^{(i)}\right)}$, for each $i$. So, from (\ref{e: disjoint union of inverse branch images}), the set $aC_{[k:k-1]}\left(\omega^{(i)}\right)+\alpha_{l_i}\left(\omega^{(i)}\right)$ is contained in $C_{k}\left(\omega^{(i)}\right)$ for all $i$. We conclude that $t$ must be contained in $\sum_{i=1}^{n}C_{k}\left(\omega^{(i)}\right)$, as needed.     

Now, suppose that $k_0>0$. Then, by the same argument as before, one can deduce that 
\begin{equation}
   C_{[\infty : k_0]}\left(\omega^{(1)} \right)+ \cdots+C_{[\infty:k_0]}\left( \omega^{(n)} \right)=[0,n].
\end{equation}

Again, by (\ref{e: products equal}), the expanding maps $E_{B_1\left(\omega^{(i)}\right)}\circ \cdots \circ E_{B_{k_0}\left(\omega^{(i)}\right)}$ have the same expansion rate (over their domains), for each $i\in\{1,\dots, n\}$. So, by using a similar linearity argument from before, we can deduce that $\sum_{i=1}^nC\left(\omega^{(i)}\right)$ is a finite union of images of $[0,n]$ under linear contraction maps of the same contraction rate. Precisely, 
\begin{equation}
    \sum_{i=1}^nC\left(\omega^{(i)}\right)=\bigcup_I\left(\tilde a [0,n]+\tilde \alpha_{I}\right),
\end{equation}
for some constants $\tilde \alpha_{I}$, $I\in \{1,\dots, 2^{\ell_1+\cdots \ell_{2k_0}}\}^n$, where 
\begin{equation}
    \tilde a:=\prod_{i=1}^{k_0}\mathrm{prod}\left(B_i\left(\omega^{(1)}\right)\right)^{-1}=\cdots =\prod_{i=1}^{k_0}\mathrm{prod}\left(B_i\left(\omega^{(n)}\right)\right)^{-1},
\end{equation}
so that this set is a disjoint union of finitely many closed intervals. 
\end{proof}

Now, we give conditions on the underlying sequences for which the $n$-sum of such Cantor sets have zero lower box-counting dimension and zero $h$-measure, when $n<n_0$. 

\begin{proposition}\label{prop: sum has zero LB dimension}
   Suppose that $\omega^{(i_1)},\cdots, \omega^{(i_n)}$, $n<n_0$, are sequences such that 
   \begin{equation}
       \omega^{(i_1)}_k=\cdots=\omega^{(i_{n})}_k=0,
   \end{equation}
for infinitely many $k$. Then, 
   \begin{equation}\label{e: zero LB for multi sum}
    \underline{\dim}_B\left(  C\left(\omega^{(i_1)}\right)+\cdots+ C\left(\omega^{(i_n)}\right)\right)=0
\end{equation}
and
\begin{equation}\label{e: h measure is zero in construction - prop}
    \mu^h\left(  C\left(\omega^{(i_1)}\right)+\cdots+ C\left(\omega^{(i_n)}\right) \right)=0.
\end{equation}
   
\end{proposition}

\begin{proof}
    Let $(k(m))$ be a sequence such that
    \begin{equation}
\omega_{k(m)}^{(i_1)}=\cdots=\omega_{k(m)}^{(i_n)}=0,
    \end{equation}
for all $m$. First, by (\ref{e: products equal}), we know that 
\begin{equation}
    \prod_{j=1}^{k(m)-1} \mathrm{prod}\left(B_j\left(\omega^{(i_1)}\right)\right)=\cdots=\prod_{j=1}^{k(m)-1}\mathrm{prod}\left( B_j\left(\omega^{(i_n)}\right)\right)=: \tilde d_{k(m)},
\end{equation}
for all $m$. Now, consider $\delta(m)$ defined by 
\begin{equation}
\delta(m)^{-1}:=(b_{k(m)})^{\ell_{2k(m)-1}}\tilde d_{k(m)}. 
\end{equation}
 For any $i=i_1,\dots, i_n$, the number of $\delta(m)$-intervals needed to cover $C\left(\omega^{(i)}\right)$ is precisely $2^{\ell_1+\cdots+\ell_{2k(m)-1}}$. So, the number of $n\delta(m)$-intervals needed to cover $C\left(\omega^{(i_1)}\right)+\cdots+ C\left(\omega^{(i_n)}\right)$, denoted $N_{n\delta(m)}$, is not greater than $2^{n(\ell_1+\cdots+\ell_{2k(m)-1})}$. Therefore,
\begin{align*}
    &\underline{\dim}_B\left(C\left(\omega^{(i_1)}\right)+\cdots+ C\left(\omega^{(i_n)}\right)\right)\\
    &\leq \lim\limits_{m\rightarrow \infty} \frac{\log N_{n\delta(m)}\left(C\left(\omega^{(i_1)}\right)+\cdots+ C\left(\omega^{(i_n)}\right)\right)}{-\log n\delta(m) }\\
    &\leq \lim\limits_{m\rightarrow \infty}\frac{n(\ell_1+\cdots+\ell_{2k(m)-1})\log2}{\log\left(\frac{1}{n}(b_{k(m)})^{\ell_{2k(m)-1}}\tilde d_{k(m)} \right)}\\
    &\leq \lim\limits_{m\rightarrow \infty}\frac{n(\ell_1+\cdots+\ell_{2k(m)-1})\log 2}{\ell_{2k(m)-1}\log b_{k(m)}-\log n}\\
    &=\lim\limits_{m\rightarrow \infty}n\log 2\left(\frac{\ell_{2k(m)-1}}{\ell_1+\cdots+\ell_{2k(m)-1}}\log b_{k(m)}-\frac{\log n}{\ell_1+\cdots+\ell_{2k(m)-1}}  \right)^{-1}\\
    &=0,
\end{align*}
due to (\ref{e: conditions on lengths of blocks}).

In a similar manner, let us use a covering for $C\left(\omega^{(i_1)}\right)+\cdots+ C\left(\omega^{(i_n)}\right)$ where each covering set has diameter $<n\delta(m)$ and there are no more than  $2^{n(\ell_1+\cdots+\ell_{2k(m)-1})}$ elements in this covering set, which is possible by the previous observation. Now, (\ref{e: h measure is zero in construction - prop}) follows from (\ref{e: gauge condition in sequqnce construction}) and taking the limit as $m\rightarrow\infty$.
\end{proof}

We are now ready to prove the main results. 

\subsection{Multi-sums of Cantor Sets}\label{ss: Multi-sums of Cantor Sets}

We now prove (II) of Theorem \ref{t.all},
and consequentially Theorem \ref{t: multi-sum lower box}, which will follow from

\begin{theorem}
Suppose $\omega^{(1)}, \dots, \omega^{(n_0)}$ are sequences of $0$'s and $1$'s satisfying
\begin{itemize}
    \item $\omega_k^{(1)}+\cdots +\omega_k^{(n_0)}=1$ for all $k$
    \item For each $(i_1,\dots, i_{n_0-1})\in \{1,\dots, n_0\}^{n_0-1}$, there are infinitely many $k$ such that 
    \begin{equation}\label{e: infinitely many terms are 0 - theorem}
\omega^{(i_1)}_k+\cdots+\omega^{(i_{n_0-1})}_k=0.
    \end{equation}
\end{itemize}
Then, $C\left(\omega^{(i_1)}\right)+\cdots+ C\left(\omega^{(i_n)}\right)$ is a Cantor set satisfying
\begin{equation}\label{e: zero LB for multi sum2}
    \underline{\dim}_B\left(  C\left(\omega^{(i_1)}\right)+\cdots+ C\left(\omega^{(i_n)}\right)\right)=0
\end{equation}
and
\begin{equation}\label{e: h measure is zero in construction}
    \mu^h\left(  C\left(\omega^{(i_1)}\right)+\cdots+ C\left(\omega^{(i_n)}\right) \right)=0,
\end{equation}
for all $n<n_0$ and $(i_1,\dots, i_n)\in \{1,\dots, n_0\}^n$, and 
\begin{equation}\label{e: sum of n sets is interval - LB}
C\left(\omega^{(1)}\right)+\cdots+ C\left(\omega^{(n_0)}\right)=[0,n_0].
\end{equation}
\end{theorem}

\begin{proof}

First, by Proposition \ref{prop: sum has zero LB dimension}, the condition (\ref{e: infinitely many terms are 0 - theorem}) implies that (\ref{e: zero LB for multi sum2}) and (\ref{e: h measure is zero in construction}) hold for any tuple $(i_1,\cdots, i_n)\in\{ 1,\cdots, n_0\}^n$ with $n<n_0$. That $C\left(\omega^{(i_1)}\right)+\cdots+ C\left(\omega^{(i_n)}\right)$ is a Cantor set follows from the fact that it does not contain any intervals, and the finite sum of any perfect compact sets is perfect. Also, (\ref{e: sum of n sets is interval - LB}) follows directly from Proposition \ref{prop: n sum is a collction of intervals}.

\end{proof}

\subsection{An Uncountable Family of Zero Dimensional Cantor Sets}\label{ss: An Uncountable Family of Zero Dimensional Cantor Sets}

We now prove (III) of Theorem \ref{t.all}, and consequentially, Theorem \ref{t: uncountable lower box}.
A main ingredient will rely on the following proposition regarding subsets of $\mathbb N$, e.g. see \cite[Lemma 9.21]{J}:

We will denote by $\mathrm{Card}(A)$, the cardinality of a set $A$. 

\begin{proposition}
    There is an uncountable collection $\mathcal A$ of subsets of $\mathbb N$ such that 
    \begin{itemize}
        \item $\mathrm{Card}(A)=\mathrm{Card}(\mathbb{N}\backslash A)=\infty$ for all $A\in \mathcal A$
        \item $\mathrm{Card}\left(A\cap A'\right)<\infty$ and $(A\Delta A')\cap [0,\max\{A\cap A'\}]=\emptyset$ for all $A,A'\in\mathcal A$ with $A\neq A'$
       \end{itemize}
\end{proposition}

This phenomenon can be understood in the context of sequences of $0$'s and $1$'s, in the following way.  
\begin{corollary}\label{corollary: collection of sequences}
    There is an uncountable collection $\mathcal S$ of sequences of $0$'s and $1$'s such that for any $\omega\in \mathcal S$, 
    \begin{equation}
        \omega_k= 0,
    \end{equation}
    for infinitely many $k$, and 
    if $\omega, \omega' \in \mathcal S$, then there exists a $k^*=k^*\left(\omega,\omega'\right)\geq 1$ such that 
\begin{equation}
    \omega_k=\omega_k'
\end{equation}
for all $k\leq k^*$, and if $k>k^*$, then 
\begin{equation}
    \omega_k+\omega_k'\geq 1.
\end{equation}

\end{corollary}

\begin{proof}
 Consider the mapping $\mathcal A\rightarrow \{0,1\}^{\mathbb N}$, $A\mapsto \omega^A,$ where
\begin{equation}
   \omega^A_k=\begin{cases}
  1 & \ \text{if} \ k\not\in A, \\
  0 & \ \text{if} \ k\in A
\end{cases}.
\end{equation}
 The mapping is injective, and we set $\mathcal S$ to be its image. 
\end{proof}

Now, (III) of Theorem \ref{t.all} follows from

\begin{theorem}\label{t: uncountable cantor sets construction}
  Let $\mathcal S$ be as in Corollary \ref{corollary: collection of sequences}. Then   
    \begin{itemize}
        \item[(1)] $\underline{\dim}_B\left(C\left(\omega \right)\right)=0$ and $\mu^h\left(C\left(\omega \right)\right)=0$ for all $\omega\in \mathcal S$

        \item[(2)] $C(\omega)+C\left(\omega'\right)$ is a disjoint union of finitely many intervals for all $\omega, \omega'\in \mathcal S$ with $\omega\neq \omega'$.
    \end{itemize}
\end{theorem}

\begin{proof}
Item (1) follows directly from Proposition \ref{prop: sum has zero LB dimension} and Item (2) follows from Proposition \ref{prop: n sum is a collction of intervals}. 
\end{proof}

\section{Concluding the Proofs}\label{Section: Concluding the Proofs}
We now finish the proof of Theorem \ref{t.all}.
\begin{proof}[Proof of Theorem \ref{t.all} (I)]
    Let $C_1, \dots, C_n$ be Cantor sets as in (II) of Theorem \ref{t.all}, and set 
    \begin{equation}
        C:=C_1\cup\cdots\cup C_n.
    \end{equation}
    Then, as $nC$ contains $C_1+\cdots C_n$, it is of course the entire interval $[0,n]$. On the other hand, for any $m<n$, we have 
    \begin{equation}
        mC=\bigcup_{(i_1,\dots, i_m)}\left( C_{i_1}+\cdots+ C_{i_m} \right)
    \end{equation}
where the union is taken over all possible tuples $(i_1,\dots, i_m)\in \{1,\dots, n\}^m$. Since $C_{i_1}+\cdots+ C_{i_m}$ is a Cantor set for any such tuple, $mC$ is again a Cantor set. Lastly, as $\mu^h$ is an outer measure, we also have 
    \begin{equation}
        \mu^h\left(m C  \right)\leq \sum_{(i_1,\cdots, i_m)}\mu^h\left( C_{i_1}+\cdots+ C_{i_m}\right) =0. 
    \end{equation}
\end{proof}

Finally, we conclude with the proof of Theorem \ref{t.2}, and consequentially that of Theorem \ref{t.first} as well. 
\begin{proof}[Proof of Theorem \ref{t.2}]
This is an application of Theorem \ref{t.all} (I) for the specific gauge function
    \begin{equation}\label{e.h}        h(x):=\begin{cases}|\log(x)|^{-1} & 0<x<1 \\ 0 & x=0\end{cases},
    \end{equation}
combined with the following result by Erd\"os and Gillis \cite{EG} (see also \cite{Lin}, \cite{Car}, and a discussion in Section \ref{ss.Haus} above): 
\begin{theorem}[Erd\"os and  Gillis, \cite{EG}]
If a set $C\subset [0,1]$ has finite Hausdorff measure with respect to the gauge function $h$ from (\ref{e.h}), then its logarithmic capacity vanishes. 
\end{theorem}
\end{proof}

\end{document}